\documentclass[reqno]{amsart}

\usepackage{amsmath,amssymb,amsfonts}
\usepackage[margin=1.15in]{geometry}
\usepackage[T1]{fontenc}
\usepackage{microtype}
\usepackage{amsthm,mathtools}
\usepackage{enumitem}
\usepackage{xcolor}
\usepackage{url}
\usepackage[colorlinks=true,linkcolor=blue!55!black,
  citecolor=blue!55!black,urlcolor=blue!55!black]{hyperref}
\usepackage[capitalize,noabbrev]{cleveref}

\hypersetup{
  pdftitle={The Noetherian Case of Bayart's Power-Series Question},
  pdfauthor={}
}

\theoremstyle{plain}
\newtheorem{theorem}{Theorem}[section]
\newtheorem{proposition}[theorem]{Proposition}
\newtheorem{lemma}[theorem]{Lemma}
\newtheorem{corollary}[theorem]{Corollary}

\theoremstyle{definition}

\newcommand{\Spec}{\operatorname{Spec}}
\newcommand{\Supp}{\operatorname{Supp}}
\newcommand{\Ass}{\operatorname{Ass}}
\newcommand{\Jac}{\operatorname{Jac}}
\newcommand{\depth}{\operatorname{depth}}
\newcommand{\Ext}{\operatorname{Ext}}
\newcommand{\Tor}{\operatorname{Tor}}
\newcommand{\Hom}{\operatorname{Hom}}
\newcommand{\Frac}{\operatorname{Frac}}
\newcommand{\htp}{\operatorname{ht}}

\title{The Noetherian Case of Bayart's Power-Series Question}

\author{Viet-Hoang Tran}
\address{Department of Mathematics, National University of Singapore, Singapore 119076}
\email{hoang.tranviet@u.nus.edu}
\urladdr{https://vh-tran.github.io/}

\author{Dung V. Nguyen}
\address{Department of Mathematics, National University of Singapore, Singapore 119076}
\email{dungnv@u.nus.edu}
\urladdr{https://kpup1710.github.io/}

\author{Quang X. Nguyen}
\address{Independent Researcher}
\email{ngxquang.work@gmail.com}

\author{Thieu N. Vo}
\address{Department of Computer Science, University of Bath, United Kingdom}
\email{ntv22@bath.ac.uk}

\author{Tan M. Nguyen}
\address{Department of Mathematics, National University of Singapore}
\email{tanmn@nus.edu.sg}
\urladdr{https://tanmnguyen89.github.io/}


\begin{document}

\begin{abstract}
Let $R$ be a commutative Noetherian ring.  We prove that if the one-variable
formal power-series ring $R[[x]]$ is a unique factorization domain, then so is
the two-variable formal power-series ring $R[[x,y]]$.  This resolves a
question raised by Bayart in 1973 for Noetherian coefficient rings.  The proof
uses the divisor theory of Noetherian normal domains,
expressed through finite rank-one reflexive modules.
\end{abstract}

\maketitle

\section{Introduction}


For a unique factorization domain (UFD) $R$, it is natural to ask whether the formal
power-series ring $R[[y]]$ is again a UFD. The subject has a long history.
Lasker obtained an early positive result in 1905 in the complete local
setting, in particular for formal power-series rings over a field
\cite{Lasker1905}.  The general question was studied systematically by Krull
in the 1930s.  He proved that \(R[[y]]\) is a UFD when \(R\) is a discrete
valuation ring with infinite residue field \cite{Krull1938}; Cohen subsequently
settled the corresponding finite residue-field case \cite{Cohen1946}.  These
results notwithstanding, unique factorization is not preserved by formal
power-series extension in general.  In 1961 Samuel constructed an
example of a UFD \(R\) for which \(R[[y]]\) is not a UFD
\cite[\S4]{Samuel1961}.  His work initiated an extensive investigation of
the UFD property in power-series and complete local rings, with subsequent
contributions by Buchsbaum, Claborn, Danilov, Salmon, and others
\cite{Buchsbaum1961,Claborn1965,Danilov1970,Salmon1964}.  In particular,
Samuel asked whether the additional hypothesis that \(R\) be a complete local
UFD would force \(R[[y]]\) to be a UFD \cite[\S4]{Samuel1961}; this too has
a negative answer, as shown by Salmon \cite{Salmon1966}.  Samuel's study of
the obstructions to unique factorization also had consequences beyond
power-series rings.  Most notably,
Grothendieck proved in SGA~2 that a Noetherian local complete-intersection
domain whose localizations at all prime ideals of height at most \(3\) are
UFDs is itself a UFD, thereby settling another conjecture of Samuel
\cite[Expos\'e~XI, Corollaire~3.14]{SGA2}; see also Hartshorne--Ogus
\cite{HartshorneOgus1974} and
Lipman \cite{Lipman1975} for further developments and perspective.

A fundamental positive result in the Noetherian theory was proved
independently by Samuel and Buchsbaum: if \(R\) is a Noetherian regular UFD,
then the formal power-series ring \(R[[y]]\) is again a UFD
\cite[Theorem~2.1]{Samuel1961} and
\cite[Theorem~3.2]{Buchsbaum1961}.
In fact, Samuel's theorem shows that \(R[[y]]\) remains a regular UFD.
Thus unique factorization is stable under adjoining a formal power-series
variable within the class of Noetherian regular UFDs.  Beyond the regular
Noetherian setting, however, the behavior of unique factorization under
formal power-series extension is considerably more subtle.

Bayart asked a different question: if $R$ is a domain and $R[[x]]$ is a UFD,
must $R[[x,y]]$ also be a UFD \cite[p.~449]{Bayart1973}?  Writing
$A=R[[x]]$, the hypothesis gives that $A$ is a UFD, but it imposes no
regularity hypothesis on $A$; consequently, the Samuel--Buchsbaum theorem does
not apply directly.  Bayart later proved a finite irreducibility theorem for
UFDs containing $\mathbb Q$
\cite[\S IV, Th\'eor\`eme~(T)]{Bayart1981}.  Jones and Paran extended this
theorem to arbitrary Krull domains and used it to solve Landweber's problem;
they also recorded in July 2026 that Bayart's question remained open
\cite[Theorems~6.4 and~7.3, and \S7]{JonesParan2026}.

Our main result gives an affirmative answer when the coefficient ring is
Noetherian.

\begin{theorem}\label{thm:bayart-noetherian}
Let $R$ be a commutative Noetherian ring.  If $R[[x]]$ is a UFD, then
$R[[x,y]]$ is a UFD.
\end{theorem}

In particular, Bayart's question has an affirmative answer throughout the
Noetherian category.  Iterating \cref{thm:bayart-noetherian}, we obtain that if
$R$ is Noetherian and $R[[x_1]]$ is a UFD, then
$R[[x_1,\ldots,x_n]]$ is a UFD for every $n\geq 1$.

Set $A=R[[x]]$ and $B=R[[x,y]]$.  The proof uses the characterization of
the UFD property in terms of finite rank-one reflexive modules.  Given such a
$B$-module $F$, we choose a prime section $z=y-x^N$ avoiding the relevant
associated primes of $\Ext^1_B(F^\vee,B)$.  Since $B/zB\cong A$, the fact that
$A$ is a UFD, together with a depth analysis, controls $F/zF$; an Ext-vanishing
and Nakayama argument then
lifts freeness to $F^\vee$, and hence to $F$.  Section~2 develops the required
reflexive-module and lifting results, and Section~3 constructs the movable
section and completes the proof.

\section{Reflexive modules and depth}

Throughout, rings are commutative with identity, ring homomorphisms preserve
identity, and modules called finite are finitely generated.  If $S$ is a
domain and $M$ is a finite $S$-module, write
\[
  M^\vee=\Hom_S(M,S).
\]
The same notation is used after localization or passage to a quotient when
the base ring is clear.  The depth of the zero module is $+\infty$.

We begin with the reflexive-module facts used below.  If $S$ is a
Noetherian normal domain with fraction field $K$, and $M$ is a finite
$S$-module, then
\begin{equation}
\label{eq:reflexive-characterizations}
\begin{aligned}
  M\text{ is reflexive}
  &\quad\Longleftrightarrow\quad
  M\text{ is torsion-free and satisfies }(S_2)\\
  &\quad\Longleftrightarrow\quad
  M\text{ is torsion-free and }
  M=\bigcap_{\htp \mathfrak p=1}M_{\mathfrak p}
       \text{ inside }M\otimes_S K.
\end{aligned}
\end{equation}
Here the intersection characterization is
\cite[Proposition~1, pp.~238--239]{Samuel1964}.  To compare it with the $(S_2)$
formulation, use the reflexivity criterion
\cite[Proposition~1.4.1]{BrunsHerzog} together with Serre's normality
criterion \cite[Theorem~39, p.~125]{MatsumuraCA}.  Indeed, over a normal domain the
local rings of height at most one are a field or a DVR, while at all other
primes the ring has depth at least two.  Thus the local conditions in
Bruns--Herzog are exactly torsion-freeness and $(S_2)$.

We shall also use the following two consequences.  Reflexivity
commutes with localization: choose a finite presentation of $M$, observe by
localizing its dual presentation that
$(M^\vee)_{\mathfrak p}\cong(M_{\mathfrak p})^\vee$, and localize the
evaluation map.  Moreover, $L^\vee$ is reflexive for every finite module
$L$ over a Noetherian normal domain.  To see this directly, dualize a finite
presentation of $L$.  This embeds $L^\vee$ in a finite free module with
torsion-free cokernel; the depth lemma
\cite[Proposition~1.2.9]{BrunsHerzog} gives depth at least two wherever the
base ring has depth at least two.  At primes where the base ring has depth at
most one, normality makes the local ring a field or a DVR; the localized dual
is torsion-free there and hence free.  The criterion
\cite[Proposition~1.4.1]{BrunsHerzog} now proves reflexivity.

We also recall the regular-element formula.  If $(C,\mathfrak m)$ is
Noetherian local, $H$ is a nonzero finite $C$-module, and
$h\in\mathfrak m$ is $H$-regular, then
\begin{equation}
\label{eq:regular-element-depth}
  \depth_C(H/hH)=\depth_C(H)-1.
\end{equation}
Indeed, apply $\Hom_C(C/\mathfrak m,-)$ to
$0\to H\xrightarrow{h}H\to H/hH\to0$.  Multiplication by $h$ is zero on
each $\Ext_C^i(C/\mathfrak m,H)$, so the long exact sequence and the
grade--Ext description of depth
\cite[Proposition~1.2.10(e)]{BrunsHerzog} give
\eqref{eq:regular-element-depth}.

\begin{proposition}[Reflexive criterion for the UFD property]
\label{prop:reflexive-ufd}
Let $D$ be a Noetherian normal domain.  Then the following are equivalent:
\begin{enumerate}[label=\textup{(\roman*)}]
  \item $D$ is a UFD;
  \item every finite rank-one reflexive $D$-module is free of rank one.
\end{enumerate}
\end{proposition}

\begin{proof}
Suppose first that $D$ is a UFD, and let $M$ be a finite rank-one reflexive
$D$-module.  A UFD is integrally closed
\cite[p.~63]{AtiyahMacdonald}.  Put $K=\Frac(D)$.  If $D$ is a field, the
assertion is immediate.  Otherwise, choose
an isomorphism $M\otimes_DK\simeq K$, and use it
to regard $M$ as a finite fractional $D$-submodule $I\subset K$.  If
$\mathfrak p$ is a height-one prime, the UFD property gives a prime element
$\pi_{\mathfrak p}$ with $\mathfrak p=(\pi_{\mathfrak p})$, and
$D_{\mathfrak p}$ is a DVR, so
\[
  I_{\mathfrak p}
  =\pi_{\mathfrak p}^{n_{\mathfrak p}}D_{\mathfrak p}
\]
for a uniquely determined integer $n_{\mathfrak p}$.

Only finitely many $n_{\mathfrak p}$ are nonzero.  Indeed, choose fractional
generators $f_1,\ldots,f_r\in K$ of $I$.  Then
\[
  n_{\mathfrak p}=\min_i v_{\mathfrak p}(f_i),
\]
and the numerator and denominator of each $f_i$ have only finitely many
prime factors in the UFD $D$.  Define
\[
  a=\prod_{\htp\mathfrak p=1}
       \pi_{\mathfrak p}^{n_{\mathfrak p}}\in K^\times;
\]
the product is finite.  Then $I_{\mathfrak p}=aD_{\mathfrak p}$ for every
height-one prime.  Both $I$ and $aD$ are reflexive, so
\eqref{eq:reflexive-characterizations} gives
\[
 I=\bigcap_{\htp\mathfrak p=1}I_{\mathfrak p}
  =\bigcap_{\htp\mathfrak p=1}aD_{\mathfrak p}
  =aD.
\]
Thus $M\simeq D$.

Conversely, assume (ii), and let $P$ be a height-one prime of $D$.  It is a
finite torsion-free module of rank one.  We verify that it is reflexive.
For every height-one prime $Q$,
\[
 P_Q=
 \begin{cases}
   PD_P,&Q=P,\\
   D_Q,&Q\neq P.
 \end{cases}
\]
Indeed, distinct height-one primes are incomparable, so for $Q\neq P$ an
element of $P\setminus Q$ becomes a unit in $D_Q$.  If
$c\in\bigcap_{\htp Q=1}P_Q$ inside $K$, then
$c\in\bigcap_{\htp Q=1}D_Q=D$ by
\eqref{eq:reflexive-characterizations}.  Also $c\in PD_P$, so some
$s\notin P$ satisfies $sc\in P$; primality gives $c\in P$.  Hence
\[
  P=\bigcap_{\htp Q=1}P_Q,
\]
and $P$ is reflexive by \eqref{eq:reflexive-characterizations}.  Hypothesis
(ii) makes $P$ free of rank one, hence principal.  Thus every height-one
prime of $D$ is principal, and $D$ is a UFD by
\cite[Theorem~47, p.~141]{MatsumuraCA}.
\end{proof}

We next isolate the three homological ingredients used in the power-series
argument.  The first lifts freeness of a dual across a regular element in the
Jacobson radical.

\begin{lemma}[Lifting freeness of a dual]
\label{lem:lifting-dual}
Let $S$ be a Noetherian ring, let $u\in\Jac(S)$ be a nonzerodivisor on
$S$, and suppose that
\[
  \overline S=S/uS
\]
is a domain.  Let $M$ be a finite $S$-module such that $u$ is a
nonzerodivisor on $M$, and put
\[
  \overline M=M/uM.
\]
Assume the following:
\begin{enumerate}[label=\textup{(\roman*)}]
  \item $\overline M$ is torsion-free over $\overline S$;
  \item $\overline M^\vee$ is a finite free $\overline S$-module;
  \item $\overline M_{\mathfrak p}$ is free over
  $\overline S_{\mathfrak p}$ for every
  $\mathfrak p\in\Spec(\overline S)$ satisfying
  $\depth\overline S_{\mathfrak p}\leq 2$.
\end{enumerate}
Then $M^\vee$ is a finite free $S$-module.
\end{lemma}

\begin{proof}
Because $\overline M$ is finite and torsion-free over the domain
$\overline S$, its evaluation map into the double dual is injective.  Indeed,
put $K=\Frac(\overline S)$ and
$V=\overline M\otimes_{\overline S}K$.  Choose a $K$-basis of $V$ and a
nonzero $d\in\overline S$ clearing the coordinates of a finite generating
set of $\overline M$.  Since $\overline M$ is torsion-free, multiplication by
$d$, followed by the coordinate isomorphism $V\cong K^r$, embeds
$\overline M$ into $\overline S^r$.  The coordinate maps are elements of
$\overline M^\vee$ and separate points, so the evaluation map is injective.
Thus there is an exact sequence of finite modules
\begin{equation}
\label{eq:lifting-cokernel}
  0\longrightarrow\overline M
   \longrightarrow\overline M^{\vee\vee}
   \longrightarrow C\longrightarrow0.
\end{equation}
The middle module is finite free, because $\overline M^\vee$ is finite
free.

Every finite module over a Noetherian ring is finitely presented, so the
finite-presentation argument given above shows that duals commute with
localization.  At every prime
$\mathfrak p$ with
$\depth\overline S_{\mathfrak p}\leq2$, hypothesis (iii) says that the
localized evaluation map is an isomorphism.  Hence $C_{\mathfrak p}=0$.
It follows that
\begin{equation}
\label{eq:support-c-depth}
 \depth\overline S_{\mathfrak p}\geq3
 \qquad\text{for every }\mathfrak p\in\Supp_{\overline S}(C).
\end{equation}
If $C=0$, the required Ext vanishing is immediate.  Otherwise, put
$I=\operatorname{Ann}_{\overline S}(C)$.  Since $C$ is finite,
$\Supp(C)=V(I)$.  The grade--support and grade--Ext formulas give
\[
 \operatorname{grade}(I,\overline S)
 =\inf_{\mathfrak p\in\Supp(C)}
     \depth\overline S_{\mathfrak p}
 =\inf\{i:\Ext^i_{\overline S}(C,\overline S)\neq0\};
\]
see \cite[Proposition~1.2.10(a),(e)]{BrunsHerzog}.  Therefore
\eqref{eq:support-c-depth} gives
\[
 \Ext^i_{\overline S}(C,\overline S)=0
 \qquad (i=0,1,2).
\]
Applying $\Hom_{\overline S}(-,\overline S)$ to
\eqref{eq:lifting-cokernel}, and using that
$\overline M^{\vee\vee}$ is free, the relevant part of the long exact
sequence is
\[
 0=\Ext^1_{\overline S}
       (\overline M^{\vee\vee},\overline S)
 \longrightarrow
 \Ext^1_{\overline S}(\overline M,\overline S)
 \longrightarrow
 \Ext^2_{\overline S}(C,\overline S)=0.
\]
Consequently
\[
 \Ext^1_{\overline S}(\overline M,\overline S)=0.
\]

The exact sequence
\[
 0\longrightarrow S\xrightarrow{\,u\,}S
 \longrightarrow\overline S\longrightarrow0
\]
is a length-one free resolution of $\overline S$.  Since $u$ is also
regular on $M$,
\[
 \Tor^S_1(M,\overline S)=\ker(u:M\to M)=0,
 \qquad
 \Tor^S_i(M,\overline S)=0\quad(i>1).
\]
Hence, if $P_\bullet\to M$ is a free resolution, then
$P_\bullet\otimes_S\overline S$ is a free resolution of
$\overline M$.  Termwise adjunction gives
\[
 \Ext^1_S(M,\overline S)
 \cong\Ext^1_{\overline S}(\overline M,\overline S)=0.
\]

Apply $\Hom_S(M,-)$ to the same short exact sequence.  The resulting long
exact sequence contains
\[
 \Ext^1_S(M,S)\xrightarrow{\,u\,}\Ext^1_S(M,S)
 \longrightarrow\Ext^1_S(M,\overline S)=0.
\]
Thus multiplication by $u$ is surjective on $\Ext^1_S(M,S)$.  This Ext
module is finite: construct a resolution of $M$ by finite free modules and
compute Ext from the resulting complex of finite modules.  Since
$u\in\Jac(S)$, Nakayama's lemma
\cite[Proposition~2.6]{AtiyahMacdonald} yields
\begin{equation}
\label{eq:ext1-S-zero}
 \Ext^1_S(M,S)=0.
\end{equation}

The preceding part of the same long exact sequence, together with
\eqref{eq:ext1-S-zero}, gives
\[
 M^\vee/uM^\vee
 \cong\Hom_S(M,\overline S).
\]
Every $S$-linear map from $M$ to $\overline S$ kills $uM$, so
\begin{equation}
\label{eq:dual-mod-u}
 M^\vee/uM^\vee
 \cong\Hom_{\overline S}(\overline M,\overline S)
 =\overline M^\vee.
\end{equation}

The module $M^\vee$ is finite because $M$ is finitely presented over the
Noetherian ring $S$.  Lift a basis of the right-hand side to elements of
$M^\vee$.  They define
an $S$-linear map
\[
 \varphi:S^r\longrightarrow M^\vee
\]
whose reduction modulo $u$ is an isomorphism.  If
$L=\operatorname{coker}\varphi$, then $L/uL=0$; the module $L$ is finite,
so Nakayama gives
$L=0$.  Thus $\varphi$ is surjective.

Let $K=\ker\varphi$, which is finite because $S$ is Noetherian.  The
element $u$ is regular on $M^\vee$: if $u\lambda=0$, then
$u\lambda(m)=0$ in $S$ for every $m\in M$, and regularity of $u$ on
$S$ gives $\lambda=0$.  Therefore
\[
 \Tor^S_1(M^\vee,\overline S)=0.
\]
Reducing
\[
 0\longrightarrow K\longrightarrow S^r
 \xrightarrow{\varphi}M^\vee\longrightarrow0
\]
modulo $u$, we obtain
\[
 0\longrightarrow K/uK\longrightarrow\overline S^r
 \xrightarrow{\overline\varphi}M^\vee/uM^\vee
 \longrightarrow0.
\]
The map $\overline\varphi$ is an isomorphism by construction and
\eqref{eq:dual-mod-u}; hence $K/uK=0$.  A second application of Nakayama
gives $K=0$, so $M^\vee\simeq S^r$.
\end{proof}

The next lemma identifies the depth-two obstruction to the local freeness
required by the lifting argument.

\begin{lemma}[Depth-two detector]
\label{lem:depth-two-detector}
Let $S$ be a Noetherian normal domain, let $M\neq0$ be a finite reflexive
$S$-module, and let $\mathfrak q\in\Spec(S)$.  If
\[
 \depth S_{\mathfrak q}\geq3,
 \qquad
 \depth M_{\mathfrak q}=2,
\]
then
\[
 \mathfrak q\in
 \Ass_S\Ext^1_S(M^\vee,S).
\]
\end{lemma}

\begin{proof}
In the present finite Noetherian setting, the finite-presentation argument
above shows that duals and reflexivity commute with localization.  Ext also
commutes with localization; this follows by localizing a resolution by finite
free modules, as in \cite[Proposition~3.3.10]{Weibel}.  More precisely,
\[
 (M^\vee)_{\mathfrak q}\cong(M_{\mathfrak q})^\vee,
\qquad
 \Ext^1_S(M^\vee,S)_{\mathfrak q}
 \cong
 \Ext^1_{S_{\mathfrak q}}
 ((M_{\mathfrak q})^\vee,S_{\mathfrak q}).
\]
It is therefore enough to work over the local ring $S_{\mathfrak q}$.
For the rest of the proof rename this local ring $S$, write its maximal
ideal as $\mathfrak m$, and rename $M_{\mathfrak q}$ as $M$.  Put
\[
 d=\depth S\geq3,
 \qquad N=M^\vee,
 \qquad \Delta=\Ext^1_S(N,S).
\]

Choose a finite free surjection $P\twoheadrightarrow N$, and let $K$ be
its kernel:
\begin{equation}
\label{eq:first-syzygy}
 0\longrightarrow K\longrightarrow P\longrightarrow N\longrightarrow0.
\end{equation}
All these modules are finite.  Dualizing \eqref{eq:first-syzygy} gives the
exact sequence
\[
 0\longrightarrow N^\vee\longrightarrow P^\vee
 \longrightarrow K^\vee\longrightarrow\Delta\longrightarrow0,
\]
because $\Ext^1_S(P,S)=0$.  Reflexivity of $M$ identifies
$N^\vee=M^{\vee\vee}$ with $M$.  If

\[
 C=\operatorname{im}(P^\vee\to K^\vee),
\]
we obtain two short exact sequences
\begin{align}
 0&\longrightarrow M\longrightarrow P^\vee
      \longrightarrow C\longrightarrow0,
 \label{eq:depth-C-one}\\
 0&\longrightarrow C\longrightarrow K^\vee
      \longrightarrow\Delta\longrightarrow0.
 \label{eq:C-K-Delta}
\end{align}

The free module $P^\vee$ has depth $d$, whereas $M$ has depth $2$.
The module $C$ is nonzero: if $C=0$, then
\eqref{eq:depth-C-one} would give $M\cong P^\vee$, contrary to these two
depth values.
The depth lemma applied to \eqref{eq:depth-C-one} gives
\begin{equation}
\label{eq:C-depth-exact}
 \depth C=1.
\end{equation}
For completeness, it first gives $\depth C\geq1$.  If
$\depth C\geq2$, its other inequality would give
\[
 2=\depth M
 \geq\min\{\depth P^\vee,\depth C+1\}
 \geq3,
\]
a contradiction.

We next prove
\begin{equation}
\label{eq:K-dual-depth}
 \depth K^\vee\geq2.
\end{equation}
Choose a finite presentation $Q_1\to Q_0\to K\to0$.  Dualizing gives
\[
 0\longrightarrow K^\vee\longrightarrow Q_0^\vee
 \longrightarrow D\longrightarrow0,
\]
where $D=\operatorname{im}(Q_0^\vee\to Q_1^\vee)$.  If $D=0$, then
$K^\vee\cong Q_0^\vee$ is free.  If $D\neq0$, then $D$ is a submodule
of the free module $Q_1^\vee$, and therefore is torsion-free.  Over the
local domain $S$, a nonzero finite torsion-free module has depth at least
one.  The depth lemma consequently yields
\[
 \depth K^\vee
 \geq\min\{\depth Q_0^\vee,\depth D+1\}
 \geq2,
\]
which proves \eqref{eq:K-dual-depth}.

If $\Delta=0$, then \eqref{eq:C-K-Delta} would identify $C$ with
$K^\vee$, contradicting \eqref{eq:C-depth-exact} and
\eqref{eq:K-dual-depth}.  If $\Delta\neq0$ and
$\depth\Delta\geq1$, then \eqref{eq:C-K-Delta} and the depth lemma would
give
\[
 \depth C\geq
 \min\{\depth K^\vee,\depth\Delta+1\}\geq2,
\]
again contradicting \eqref{eq:C-depth-exact}.  Hence
$\depth\Delta=0$.  A finite module over a Noetherian local ring has depth
zero precisely when the maximal ideal is associated, so
$\mathfrak m\in\Ass_S(\Delta)$.  Undoing localization gives
\[
 \mathfrak q\in\Ass_S\Ext^1_S(M^\vee,S),
\]
as required.  Here the depth lemma is
\cite[Proposition~1.2.9]{BrunsHerzog}.  For clarity, the other two facts can
be recovered as follows.  The associated primes of a finite module are
finite, and their union is its set of zero divisors
\cite[Chapter~IV, \S1, no.~1, Corollary~2 to Proposition~2, and no.~4,
Theorems~1--2 and Corollary]{BourbakiCA}.  Thus, in a local ring, depth zero
is equivalent by finite prime avoidance to the maximal ideal being
associated.  Associated primes commute with localization in the Noetherian
case by \cite[Chapter~IV, \S1, no.~2, Proposition~5 and its
Corollary]{BourbakiCA}; applying that result to the original prime
$\mathfrak q$ undoes the localization.
\end{proof}

Finally, reduction of a rank-one reflexive module modulo a prime element
preserves torsion-freeness and rank.

\begin{lemma}[Reduction modulo a prime element]
\label{lem:torsion-free-section}
Let $S$ be a Noetherian normal domain, let $u\in S$ be a nonzero prime
element, and let $F$ be a finite rank-one reflexive $S$-module.  Then
\[
 F/uF
\]
is a finite torsion-free rank-one module over $S/uS$.
\end{lemma}

\begin{proof}
The module $F$ is torsion-free, so multiplication by $u$ is injective on
$F$.  The ideal $(u)$ is a nonzero prime.  The chain
$(0)\subsetneq(u)$ gives $\htp(u)\geq1$, while the principal ideal theorem
gives $\htp(u)\leq1$ \cite[Corollary~11.17]{AtiyahMacdonald}.  Hence
$\htp(u)=1$.

It remains to prove torsion-freeness after reduction.  Let
$a\notin(u)$, and suppose that $f,h\in F$ satisfy
\[
 af=uh.
\]
Inside the generic fiber $F\otimes_S\Frac(S)$, set $w=f/u$.  We claim
that $w\in F_{\mathfrak p}$ for every height-one prime $\mathfrak p$.
If $\mathfrak p=(u)$, then $a$ is a unit in $S_{\mathfrak p}$, and
\[
 w=h/a\in F_{\mathfrak p}.
\]
If $\mathfrak p\neq(u)$, then $u\notin\mathfrak p$: otherwise the inclusion
$(u)\subseteq\mathfrak p$ between height-one prime ideals would be an
equality.  Thus $u$ is a unit in $S_{\mathfrak p}$, and
\[
 w=f/u\in F_{\mathfrak p}.
\]
The intersection characterization \eqref{eq:reflexive-characterizations}
therefore gives $w\in F$.  Hence $f=uw\in uF$.  This says exactly that
$F/uF$ is torsion-free over $S/uS$.

Normality implies that $S_{(u)}$ is a DVR.  The finite
torsion-free rank-one $S_{(u)}$-module $F_{(u)}$ is therefore free of rank
one.  At the generic point of $\Spec(S/uS)$,
\[
 (F/uF)_{(0)}
 \cong F_{(u)}/uF_{(u)}
 \cong S_{(u)}/uS_{(u)}
 \cong\Frac(S/uS).
\]
Thus $F/uF$ has rank one over $S/uS$.
\end{proof}

\section{Movable prime sections}

For the rest of the paper, let $R$ satisfy the hypotheses of
\cref{thm:bayart-noetherian}, and put
\[
 A=R[[x]],\qquad B=R[[x,y]].
\]

Since $A$ is a UFD, it is a normal domain
\cite[p.~63]{AtiyahMacdonald}.  The inclusion of constant
series $R\hookrightarrow A$ shows first that $R$ is a domain.  We also
need that $R$ is normal.  Let $K=\Frac(R)$, and let $c\in K$ be integral
over $R$.  The same monic equation shows that $c$ is integral over $A$,
and $K\subseteq\Frac(A)$.  Normality of $A$ therefore gives $c\in A$.
Under the coefficientwise inclusion
\[
 A=R[[x]]\hookrightarrow K[[x]],
\]
the element $c$ is the constant series $c+0x+0x^2+\cdots$.  Membership
in $R[[x]]$ forces its constant coefficient $c$ to belong to $R$.  Thus
$R$ is a Noetherian normal domain.

By successive applications of the formal power-series theorem for Noetherian
rings \cite[Corollary~10.27]{AtiyahMacdonald}, the rings $A$ and $B$ are
Noetherian.  For normality of $B$, recall that a
Noetherian normal domain is completely integrally closed, and that a formal
power-series ring over a completely integrally closed domain is again
completely integrally closed
\cite[Chapter~V, \S1, no.~1, Corollary to Proposition~1, p.~304; no.~4,
pp.~312--313, especially Proposition~14]{BourbakiCA}.  For the first
assertion, suppose that $c$ is almost integral over a Noetherian normal
domain $C$.  Then $C[c]$ is a submodule of $d^{-1}C$ for
some $0\neq d\in C$; it is therefore finite over $C$, so $c$ is integral
and hence belongs to $C$.  Applying the quoted complete-integral-closure
result successively in the variables $x$ and $y$, we conclude that
$B=R[[x,y]]$ is a Noetherian normal domain.

By \cref{prop:reflexive-ufd}, it remains to show that every finite
rank-one reflexive $B$-module is free.  The following lemma supplies a prime
section avoiding the obstruction detected by
\cref{lem:depth-two-detector}.

\begin{lemma}[A movable prime section]
\label{lem:movable-section}
For every finite $B$-module $T$, there is an integer $N\geq1$ such that,
with $z=y-x^N$, the following hold:
\begin{enumerate}[label=\textup{(\roman*)}]
  \item $z\in\Jac(B)$;
  \item $z$ is a nonzero prime element and $B/zB\cong A$;
  \item $z\notin\mathfrak q$ for every
        $\mathfrak q\in\Ass_B(T)\setminus V(x,y)$.
\end{enumerate}
\end{lemma}

\begin{proof}

For every integer $n\geq1$, put
\[
 z_n=y-x^n.
\]
We claim that every prime $\mathfrak q\not\supseteq(x,y)$ contains at most
one of the elements $z_n$.  Indeed, if $n<m$ and $z_n,z_m\in\mathfrak q$,
then
\[
 x^n-x^m=x^n(1-x^{m-n})\in\mathfrak q.
\]
The element $1-x^{m-n}$ is a unit, with inverse
\[
 1+x^{m-n}+x^{2(m-n)}+\cdots.
\]
Hence $x\in\mathfrak q$, and then $y=z_n+x^n\in\mathfrak q$, contrary
to the choice of $\mathfrak q$.

The finite module $T$ has only finitely many associated primes
\cite[Chapter~IV, \S1, no.~4, Theorems~1--2 and
Corollary]{BourbakiCA}.  From
the preceding paragraph, each member of
\[
 \Ass_B(T)\setminus V(x,y)
\]
forbids at most one positive integer $n$.  We may therefore choose an
integer $N\geq1$ such that
\[
 z:=y-x^N
 \quad\text{belongs to no prime in }
 \Ass_B(T)\setminus V(x,y).
\]

We next verify that $z$ is a prime element and that $B/zB\cong A$.
For $f=\sum_{i,j\geq0}a_{ij}x^iy^j$, define
\[
 \sigma_N(f)=\sum_{i,j\geq0}a_{ij}x^i(y+x^N)^j.
\]
For fixed $p,q$, the coefficient of $x^py^q$ in this expression is
\[
 \sum_{\substack{j\geq q\\i+N(j-q)=p}}
 \binom jq a_{ij}.
\]
Only the indices
\[
 q\leq j\leq q+\left\lfloor\frac pN\right\rfloor
\]
can occur, so this is a finite sum.  Thus the substitution is well-defined.
It is an $R$-algebra homomorphism: one may verify this on every finite
$(x,y)$-adic truncation, where it is ordinary polynomial substitution, and
then pass to the inverse limit.  Substitution $y\mapsto y-x^N$ gives its
inverse.  Since $\sigma_N(y-x^N)=y$,
\begin{equation}
\label{eq:section-quotient}
 B/zB\cong B/yB\cong R[[x]]=A.
\end{equation}
The ring $A$ is a domain, so $z$ is a nonzero prime element of $B$.

Finally,
\[
 z\in\Jac(B).
\]
Indeed, $x,y\in\Jac(B)$: for every $g\in B$, the coefficientwise
convergent geometric series give
\[
 (1-xg)^{-1}=\sum_{n\geq0}(xg)^n,
 \qquad
 (1-yg)^{-1}=\sum_{n\geq0}(yg)^n.
\]
Since the Jacobson radical is an ideal, it contains $(x,y)$, and hence it
contains $y-x^N$.
\end{proof}

We apply the moving-section lemma to the Ext module that records the
possible depth-two failure of a reflexive module.

\begin{proposition}[Low-depth freeness on a movable section]
\label{prop:section-freeness}
Let $F$ be a finite rank-one reflexive $B$-module, put
\[
  \Delta=\Ext^1_B(F^\vee,B),
\]
choose $z=y-x^N$ as in \cref{lem:movable-section} for $T=\Delta$, and set
$E=F/zF$.  Then:
\begin{enumerate}[label=\textup{(\roman*)}]
  \item $E$ is finite, torsion-free, and of rank one, and
        $E^\vee\cong A$;
  \item $E_x\cong A_x$;
  \item if $\mathfrak p\in\Spec(A)$ contains $x$ and
        $\depth A_{\mathfrak p}\leq2$, then
        $E_{\mathfrak p}\cong A_{\mathfrak p}$.
\end{enumerate}
Consequently, $E_{\mathfrak p}$ is free for every
$\mathfrak p\in\Spec(A)$ with $\depth A_{\mathfrak p}\leq2$.
\end{proposition}

\begin{proof}
The module $\Delta$ is finite: resolve the finite module $F^\vee$ by
finite free $B$-modules and compute Ext from the resulting complex of
finite modules.  Thus \cref{lem:movable-section} applies.

Set
\[
 E=F/zF,
\]
viewed through \eqref{eq:section-quotient} as an $A$-module.  By
\cref{lem:torsion-free-section}, $E$ is a finite torsion-free rank-one
$A$-module.  Its dual
\[
 E^\vee=\Hom_A(E,A)
\]
is finite, reflexive, and rank one: duals of finite modules are reflexive over
a Noetherian normal domain, and the rank assertion follows after tensoring
with $\Frac(A)$.  Since $A$ is a UFD,
\cref{prop:reflexive-ufd} yields
\begin{equation}
\label{eq:E-dual-free}
 E^\vee\cong A.
\end{equation}

We next prove $E_x\cong A_x$.

Let $\mathfrak s\in\Spec(A)$ satisfy
\[
 x\notin\mathfrak s,
 \qquad \htp\mathfrak s\geq2,
\]
and let $\mathfrak q$ be its inverse image under the quotient map
$B\to B/zB\cong A$.  Then $z\in\mathfrak q$ and $x\notin\mathfrak q$,
so $\mathfrak q\not\supseteq(x,y)$.  Because $\mathfrak q$ contains the
chosen element $z$, \cref{lem:movable-section}(iii) gives
\begin{equation}
\label{eq:q-not-ass}
 \mathfrak q\notin\Ass_B(\Delta).
\end{equation}

The normal domain $A$ satisfies $(S_2)$.  Since
$\dim A_{\mathfrak s}=\htp\mathfrak s\geq2$,
\[
 \depth A_{\mathfrak s}\geq2.
\]
The element $z$ is regular on $B_{\mathfrak q}$, belongs to its maximal
ideal, and
\[
 B_{\mathfrak q}/zB_{\mathfrak q}\cong A_{\mathfrak s}.
\]
Therefore, by the depth formula for reduction modulo a regular element,
\[
 \depth B_{\mathfrak q}
 =\depth A_{\mathfrak s}+1\geq3.
\]
Because $F$ is reflexive over the normal domain $B$, it satisfies
$(S_2)$; hence
\[
 \depth F_{\mathfrak q}\geq2.
\]
If equality held, \cref{lem:depth-two-detector} would imply
\[
 \mathfrak q\in\Ass_B\Ext^1_B(F^\vee,B)
 =\Ass_B(\Delta),
\]
contrary to \eqref{eq:q-not-ass}.  Here we use explicitly the finite-module
localization isomorphism
\[
 \Delta_{\mathfrak q}\cong
 \Ext^1_{B_{\mathfrak q}}
 ((F_{\mathfrak q})^\vee,B_{\mathfrak q})
\]
and the localization criterion for associated primes.  Thus
\[
 \depth F_{\mathfrak q}\geq3.
\]
Since $z$ is also regular on the torsion-free module $F_{\mathfrak q}$,
reduction by $z$ gives
\begin{equation}
\label{eq:E-depth-away-x}
 \depth E_{\mathfrak s}
 =\depth F_{\mathfrak q}-1\geq2.
\end{equation}

If $\mathfrak s$ has height one and does not contain $x$, then
$A_{\mathfrak s}$ is a DVR, and the finite torsion-free rank-one module
$E_{\mathfrak s}$ is free.  At the height-zero prime, $E$ becomes a
one-dimensional vector space over $\Frac(A)$.  Primes of $A_x$ correspond
to primes $\mathfrak s$ of $A$ not containing $x$, and under this
correspondence the local rings and heights agree:
$(A_x)_{\mathfrak sA_x}=A_{\mathfrak s}$ and
$\htp(\mathfrak sA_x)=\htp(\mathfrak s)$.  Together with
\eqref{eq:E-depth-away-x}, this shows that $E_x$ is torsion-free and
satisfies $(S_2)$ over the normal domain $A_x$.  Hence $E_x$ is
reflexive.  Duals commute with localization, so
\[
 (E_x)^\vee\cong(E^\vee)_x\cong A_x
\]
by \eqref{eq:E-dual-free}.  Taking another dual and using reflexivity gives
\begin{equation}
\label{eq:E-free-away-x}
 E_x\cong(E_x)^{\vee\vee}\cong A_x.
\end{equation}

It remains to consider primes $\mathfrak p$ containing $x$ and satisfying
$\depth A_{\mathfrak p}\leq2$.

Let $\mathfrak p\in\Spec(A)$ satisfy
\[
 x\in\mathfrak p,
 \qquad \depth A_{\mathfrak p}\leq2,
\]
and put $\mathfrak a=\mathfrak p\cap R$.  Since $A/(x)\cong R$, primes
containing $x$ correspond to primes of $R$, and
\begin{equation}
\label{eq:p-a-x}
 \mathfrak p=(\mathfrak a,x),
 \qquad
 A_{\mathfrak p}/xA_{\mathfrak p}\cong R_{\mathfrak a}.
\end{equation}
The element $x$ is regular on $A_{\mathfrak p}$, so
\begin{equation}
\label{eq:depth-Ra}
 \depth R_{\mathfrak a}
 =\depth A_{\mathfrak p}-1\leq1.
\end{equation}
The ring $R_{\mathfrak a}$ is a Noetherian normal local domain.  If it had
dimension at least two, $(S_2)$ would force its depth to be at least two,
contrary to \eqref{eq:depth-Ra}.  Thus
\[
 \dim R_{\mathfrak a}\leq1.
\]
If the dimension is zero, $R_{\mathfrak a}$ is a field.  If it is one,
normality makes $R_{\mathfrak a}$ a DVR
\cite[Proposition~9.2]{AtiyahMacdonald}.

Let $\mathfrak q$ be the inverse image of $\mathfrak p$ in $B$ under
$B\to B/zB\cong A$.  By \eqref{eq:p-a-x},
\[
 \mathfrak q=(\mathfrak a,x,y-x^N)
             =(\mathfrak a,x,y).
\]
We prove directly, without using catenarity or a dimension formula, that
$B_{\mathfrak q}$ is regular.

Suppose first that $R_{\mathfrak a}$ is a field.  Since $R$ is a domain,
this forces $\mathfrak a=(0)$.  Hence the maximal ideal of
$B_{\mathfrak q}$ is generated by $x,y$.  The chain
\[
 (0)\subsetneq(x)B_{\mathfrak q}
 \subsetneq(x,y)B_{\mathfrak q}
\]
is a strict chain of prime ideals: $B$ is a domain and
$B/(x)\cong R[[y]]$ is a domain; distinct primes contained in
$\mathfrak q$ remain distinct after localization.  Consequently
\[
 2\leq\dim B_{\mathfrak q}
 \leq\operatorname{edim}B_{\mathfrak q}\leq2.
\]
Thus equality holds and $B_{\mathfrak q}$ is regular.

Suppose next that $R_{\mathfrak a}$ is a DVR, and let $\pi$ be a
uniformizer.  Since every element of $R\setminus\mathfrak a$ becomes a unit
in $B_{\mathfrak q}$, the map $R\to B_{\mathfrak q}$ factors through
$R_{\mathfrak a}$, and
\[
 \mathfrak aB_{\mathfrak q}=\pi B_{\mathfrak q}.
\]
Thus the maximal ideal of $B_{\mathfrak q}$ is generated by
$\pi,x,y$.  There is a strict chain
\begin{equation}
\label{eq:prime-chain-DVR}
 (0)\subsetneq\mathfrak aB_{\mathfrak q}
 \subsetneq(\mathfrak a,x)B_{\mathfrak q}
 \subsetneq(\mathfrak a,x,y)B_{\mathfrak q}.
\end{equation}
To justify primality before localization, note that $\mathfrak a$ is
finitely generated because $R$ is Noetherian, and coefficientwise reduction
gives
\[
 B/\mathfrak aB\cong(R/\mathfrak a)[[x,y]],
 \qquad
 B/(\mathfrak a,x)B\cong(R/\mathfrak a)[[y]].
\]
Both quotient rings are domains.  The ideals in
\eqref{eq:prime-chain-DVR} are distinct primes contained in
$\mathfrak q$, so their localized chain is still strict.  Therefore
\[
 3\leq\dim B_{\mathfrak q}
 \leq\operatorname{edim}B_{\mathfrak q}\leq3,
\]
and $B_{\mathfrak q}$ is regular in this case as well.  For a Noetherian
local ring, the dimension is at most the minimal number of generators of its
maximal ideal, namely its embedding dimension, and equality is the
regular-local condition
\cite[Corollary~11.15 and Theorem~11.22]{AtiyahMacdonald}.  Thus no
catenarity is hidden in this argument.

A regular local ring is a UFD
\cite[Theorem~48, p.~142]{MatsumuraCA}.  The localized
module $F_{\mathfrak q}$ is finite, reflexive, and rank one, so
\cref{prop:reflexive-ufd} gives
\[
 F_{\mathfrak q}\cong B_{\mathfrak q}.
\]
Reducing modulo $z$, and using
$B_{\mathfrak q}/zB_{\mathfrak q}\cong A_{\mathfrak p}$, we conclude that
\begin{equation}
\label{eq:E-free-over-x}
 E_{\mathfrak p}
 \cong F_{\mathfrak q}/zF_{\mathfrak q}
 \cong A_{\mathfrak p}.
\end{equation}
If $x\notin\mathfrak p$, then \eqref{eq:E-free-away-x} localizes to give
$E_{\mathfrak p}\cong A_{\mathfrak p}$.  Together with
\eqref{eq:E-free-over-x}, this proves the final assertion and completes the
proof.
\end{proof}

\begin{corollary}[Reflexive modules over $B$]
\label{cor:reflexive-B-free}
Every finite rank-one reflexive $B$-module is free of rank one.
\end{corollary}

\begin{proof}
Let $F$ be such a module and put
\[
  \Delta=\Ext^1_B(F^\vee,B).
\]
As in the proof of \cref{prop:section-freeness}, the module $\Delta$ is
finite.  Choose $z=y-x^N$ by \cref{lem:movable-section} for
$T=\Delta$, and put $E=F/zF$.  By
\cref{prop:section-freeness}, the module $E$ is torsion-free,
$E^\vee\cong A$, and $E_{\mathfrak p}$ is free whenever
$\depth A_{\mathfrak p}\leq2$.

The other hypotheses of \cref{lem:lifting-dual} follow from
\cref{lem:movable-section}: the element $z$ belongs to $\Jac(B)$, is
regular on the domain $B$, and satisfies $B/zB\cong A$; it is also regular
on the torsion-free module $F$.  Apply \cref{lem:lifting-dual} with
\[
 S=B,\qquad u=z,\qquad M=F,
 \qquad\overline S=A,\qquad\overline M=E.
\]
The lemma shows that $F^\vee$ is free.  Localization at the generic point gives
\[
 F^\vee\otimes_B\Frac(B)
 \cong\Hom_{\Frac(B)}
       (F\otimes_B\Frac(B),\Frac(B))
 \cong\Frac(B),
\]
so $F^\vee$ has rank one and is isomorphic to $B$.  Reflexivity now gives
\[
  F\cong F^{\vee\vee}\cong B.
\]
\end{proof}

\begin{proof}[Proof of \cref{thm:bayart-noetherian}]
The ring $B$ is a Noetherian normal domain, and
\cref{cor:reflexive-B-free} shows that every finite rank-one reflexive
$B$-module is free.  Hence $B=R[[x,y]]$ is a UFD by
\cref{prop:reflexive-ufd}.
\end{proof}

\section*{Acknowledgments}

We used GPT-5.6 Sol, Claude Fable 5, and an agentic harness built around these
models to assist with literature search, hypothesis testing, exploration and
elimination of potential approaches, wording refinement, and manuscript
proofreading.  We are very grateful to Hieu M. Vu for technical support in the
use of the AI tools and agentic system.  The authors independently verified and
streamlined the proofs and were responsible for composing the manuscript.

\raggedbottom
\begingroup
\raggedright
\bibliographystyle{plain}
\bibliography{example_paper}

@article{Samuel1961,
  author  = {Samuel, Pierre},
  title   = {On Unique Factorization Domains},
  journal = {Illinois Journal of Mathematics},
  volume  = {5},
  number  = {1},
  year    = {1961},
  pages   = {1--17},
  doi     = {10.1215/ijm/1255629643},
  note    = {\url{https://doi.org/10.1215/ijm/1255629643}}
}

@article{Buchsbaum1961,
  author  = {Buchsbaum, David A.},
  title   = {Some Remarks on Factorization in Power Series Rings},
  journal = {Journal of Mathematics and Mechanics},
  volume  = {10},
  number  = {5},
  year    = {1961},
  pages   = {749--753},
  doi     = {10.1512/iumj.1961.10.10052},
  note    = {\url{https://iumj.org/article/1362/}}
}

@article{Bayart1973,
  author  = {Bayart, Marc},
  title   = {S{\'e}ries formelles sur un anneau factoriel},
  journal = {Comptes Rendus de l'Acad{\'e}mie des Sciences de Paris,
             S{\'e}rie A},
  volume  = {277},
  number  = {11},
  year    = {1973},
  pages   = {449--450},
  note    = {\url{https://gallica.bnf.fr/ark:/12148/bpt6k6312609m/f107.item}}
}

@incollection{Bayart1981,
  author    = {Bayart, Marc},
  title     = {Factorialit{\'e} et s{\'e}ries formelles irr{\'e}ductibles {II}},
  booktitle = {S{\'e}minaire d'Alg{\`e}bre Paul Dubreil et Marie-Paule
               Malliavin: Proceedings, Paris 1980 (33e ann{\'e}e)},
  series    = {Lecture Notes in Mathematics},
  volume    = {867},
  publisher = {Springer-Verlag},
  address   = {Berlin},
  year      = {1981},
  pages     = {174--213},
  doi       = {10.1007/BFb0090387},
  note      = {\url{https://doi.org/10.1007/BFb0090387}}
}

@misc{JonesParan2026,
  author       = {Jones, Adam and Paran, Elad},
  title        = {On {Landweber's} Unique Factorization Problem},
  year         = {2026},
  howpublished = {arXiv:2607.03475v2 [math.AC]},
  note         = {\url{https://arxiv.org/abs/2607.03475}}
}

@article{Samuel1964,
  author  = {Samuel, Pierre},
  title   = {Anneaux gradu{\'e}s factoriels et modules r{\'e}flexifs},
  journal = {Bulletin de la Soci{\'e}t{\'e} Math{\'e}matique de France},
  volume  = {92},
  year    = {1964},
  pages   = {237--249},
  doi     = {10.24033/bsmf.1608},
  note    = {\url{https://doi.org/10.24033/bsmf.1608}}
}

@book{BrunsHerzog,
  author    = {Bruns, Winfried and Herzog, J{\"u}rgen},
  title     = {{Cohen--Macaulay Rings}},
  series    = {Cambridge Studies in Advanced Mathematics},
  volume    = {39},
  edition   = {Revised},
  publisher = {Cambridge University Press},
  address   = {Cambridge},
  year      = {1998},
  doi       = {10.1017/CBO9780511608681},
  note      = {\url{https://doi.org/10.1017/CBO9780511608681}}
}

@book{AtiyahMacdonald,
  author    = {Atiyah, Michael F. and Macdonald, Ian G.},
  title     = {Introduction to Commutative Algebra},
  publisher = {Addison--Wesley},
  address   = {Reading, Massachusetts},
  year      = {1969},
  isbn      = {0-201-00361-9},
  note      = {\url{https://catalogue.bnf.fr/ark:/12148/cb373622871}}
}

@book{BourbakiCA,
  author    = {Bourbaki, Nicolas},
  title     = {Commutative Algebra: Chapters 1--7},
  series    = {Elements of Mathematics},
  publisher = {Springer-Verlag},
  address   = {Berlin},
  year      = {1989},
  note      = {English translation; softcover reprint, 1998;
               \url{https://link.springer.com/book/9783540642398}}
}

@book{MatsumuraCA,
  author    = {Matsumura, Hideyuki},
  title     = {Commutative Algebra},
  edition   = {Second},
  publisher = {Benjamin/Cummings},
  address   = {Reading, Massachusetts},
  year      = {1980},
  isbn      = {0-8053-7026-9},
  note      = {\url{https://search.worldcat.org/title/Commutative-algebra/oclc/638535177}}
}

@book{Weibel,
  author    = {Weibel, Charles A.},
  title     = {An Introduction to Homological Algebra},
  series    = {Cambridge Studies in Advanced Mathematics},
  volume    = {38},
  publisher = {Cambridge University Press},
  address   = {Cambridge},
  year      = {1994},
  doi       = {10.1017/CBO9781139644136},
  note      = {\url{https://doi.org/10.1017/CBO9781139644136}}
}

@article{Lasker1905,
  author  = {Lasker, Emanuel},
  title   = {{Zur Theorie der Moduln und Ideale}},
  journal = {Mathematische Annalen},
  volume  = {60},
  year    = {1905},
  pages   = {20--116},
  doi     = {10.1007/BF01447495},
  note    = {\url{https://doi.org/10.1007/BF01447495}}
}

@article{Krull1938,
  author  = {Krull, Wolfgang},
  title   = {{Beitr{\"a}ge zur Arithmetik kommutativer Integrit{\"a}tsbereiche.
             V.~Potenzreihenringe}},
  journal = {Mathematische Zeitschrift},
  volume  = {43},
  year    = {1938},
  pages   = {768--782},
  doi     = {10.1007/BF01181119},
  note    = {\url{https://doi.org/10.1007/BF01181119}}
}

@article{Cohen1946,
  author  = {Cohen, Irvin S.},
  title   = {On the structure and ideal theory of complete local rings},
  journal = {Transactions of the American Mathematical Society},
  volume  = {59},
  number  = {1},
  year    = {1946},
  pages   = {54--106},
  doi     = {10.1090/S0002-9947-1946-0016094-3},
  note    = {\url{https://doi.org/10.1090/S0002-9947-1946-0016094-3}}
}

@article{Salmon1964,
  author  = {Salmon, Paolo},
  title   = {Sur les s{\'e}ries formelles restreintes},
  journal = {Bulletin de la Soci{\'e}t{\'e} Math{\'e}matique de France},
  volume  = {92},
  year    = {1964},
  pages   = {385--410},
  doi     = {10.24033/bsmf.1613},
  note    = {\url{https://doi.org/10.24033/bsmf.1613}}
}

@article{Salmon1966,
  author  = {Salmon, Paolo},
  title   = {Su un problema posto da {P.~Samuel}},
  journal = {Atti della Accademia Nazionale dei Lincei. Rendiconti della
             Classe di Scienze Fisiche, Matematiche e Naturali},
  series  = {8},
  volume  = {40},
  number  = {5},
  year    = {1966},
  pages   = {801--803},
  note    = {\url{https://www.bdim.eu/item?id=RLINA_1966_8_40_5_801_0}}
}

@article{Claborn1965,
  author  = {Claborn, Luther},
  title   = {Note generalizing a result of {Samuel's}},
  journal = {Pacific Journal of Mathematics},
  volume  = {15},
  number  = {3},
  year    = {1965},
  pages   = {805--808},
  note    = {\url{https://msp.org/pjm/1965/15-3/pjm-v15-n3-p07-p.pdf}}
}

@article{Danilov1970,
  author  = {Danilov, V. I.},
  title   = {On a conjecture of {Samuel}},
  journal = {Mathematics of the {USSR}-Sbornik},
  volume  = {10},
  number  = {1},
  year    = {1970},
  pages   = {127--137},
  doi     = {10.1070/SM1970v010n01ABEH001590},
  note    = {\url{https://doi.org/10.1070/SM1970v010n01ABEH001590}}
}

@book{SGA2,
  author    = {Grothendieck, A. and Raynaud, M.},
  title     = {Cohomologie locale des faisceaux coh{\'e}rents et th{\'e}or{\`e}mes de Lefschetz locaux et globaux ({SGA} 2)},
  series    = {Advanced Studies in Pure Mathematics},
  volume    = {2},
  publisher = {North-Holland},
  address   = {Amsterdam},
  year      = {1968},
  note      = {Updated edition: \url{https://arxiv.org/abs/math/0511279}}
}

@article{HartshorneOgus1974,
  author  = {Hartshorne, Robin and Ogus, Arthur},
  title   = {On the factoriality of local rings of small embedding codimension},
  journal = {Communications in Algebra},
  volume  = {1},
  number  = {5},
  year    = {1974},
  pages   = {415--437},
  doi     = {10.1080/00927877408548627},
  note    = {\url{https://doi.org/10.1080/00927877408548627}}
}

@inproceedings{Lipman1975,
  author    = {Lipman, Joseph},
  title     = {Unique factorization in complete local rings},
  booktitle = {Algebraic Geometry (Proc. Sympos. Pure Math., Vol. 29, Humboldt State Univ., Arcata, Calif., 1974)},
  publisher = {American Mathematical Society},
  address   = {Providence, RI},
  year      = {1975},
  pages     = {531--546},
  doi       = {10.1090/pspum/029/0374125},
  note      = {\url{https://doi.org/10.1090/pspum/029/0374125}}
}
\endgroup

\end{document}